\documentclass[12pt]{amsart}
\usepackage{amssymb}
\usepackage[latin1]{inputenc}
\usepackage{amsthm}
\usepackage{amsfonts}
\usepackage{mathrsfs}
\usepackage{graphicx}
\usepackage{amsmath}
\usepackage[all]{xy}

\let\mathcal\mathscr

\def\bC{{\mathbb C}}
\def\bD{{\mathbb D}}

\def\bP{{\mathbb P}}
\def\bN{{\mathbb N}}
\def\bQ{{\mathbb Q}}
\newtheorem{thm}{Theorem}[section]

\def\N{{\bf N}}

\def\tilde{\widetilde}

\def\phi{\varphi}

\numberwithin{equation}{section}

\newtheorem{theorem}[thm]{Theorem}
\newtheorem*{thma}{Theorem A}
\newtheorem*{thmb}{Theorem B}
\newtheorem*{thmc}{Theorem C}

\newtheorem{conjecture}[thm]{Conjecture}
\newtheorem{corollary}[thm]{Corollary}

\newtheorem{definition}[thm]{Definition}
\newtheorem{example}[thm]{Example}

\newtheorem{problem}[thm]{Problem}
\newtheorem{proposition}[thm]{Proposition}

\newtheorem{question}[thm]{Question}

\newtheorem{remark}[thm]{Remark}

\begin{document}

\title {Degeneracy of holomorphic maps via orbifolds}
\author{Erwan Rousseau}
\date{}

\begin{abstract}
We use orbifold structures to deduce degeneracy statements for holomorphic maps into logarithmic surfaces. We improve former results in the smooth case and generalize them to singular pairs. In particular, we give applications on nodal surfaces and complements of singular plane curves. 
\end{abstract}

\maketitle

\section{Introduction}
It is now classical that the properties of holomorphic maps in compact complex manifolds are closely related to the properties of the canonical line bundle. More precisely, one can expect following Green-Griffiths that the following is true
\begin{conjecture}
Let $X$ be a projective manifold of general type, i.e. its canonical line bundle $K_X$ is big. Then there exists a proper subvariety $Y\subsetneq X$ which contains every non-constant entire curve $f:\bC \to X$.
\end{conjecture}

It can be observed that positivity properties of the canonical bundle can be generalized to a more general situation than the usual compact setting, and still give properties of degeneracy for holomorphic maps. For example, a classical result of Nevanlinna is
\begin{theorem}[\cite{Nev}]
Let $a_1,\dots,a_k \in \bP^1$ and $m_1,\dots,m_k \in \N\cup{\infty}$. If 
$$\sum_{i=1}^k\left(1-\frac{1}{m_i}\right) > 2,$$
then every entire curve $f:\bC\to \bP^1$ which is ramified over $a_i$ with multiplicity at least $m_i$ is constant.
\end{theorem}

Following Green-Griffiths' philosophy, this degeneracy property should correspond to the positivity property of some canonical line bundle. Here one easily observes that the right canonical line bundle to consider is
$$K_{\bP^1}+\sum_{i=1}^k\left(1-\frac{1}{m_i}\right)a_i,$$
which can be seen as the canonical line bundle of the pair $(\bP^1,\Delta)$ where $\Delta=\sum_{i=1}^k\left(1-\frac{1}{m_i}\right)a_i.$

More generally, following Campana \cite{C04}, a pair $(X,\Delta)$, consisting of a complex manifold $X$ and a $\mathbb{Q}$-divisor $\Delta=\sum_{i=1}^k\left(1-\frac{1}{m_i}\right)Z_i$, is called a geometric orbifold. Positivity properties of the orbifold canonical line bundle, $K_X+\Delta$, should provide degeneracy statements for orbifold entire curves $f:\bC \to X$, i.e. entire curves which ramifies with multiplicity at least $m_i$ over $Z_i$. 

In this paper we shall study the case of surfaces improving and generalizing results of a previous work \cite{Rou08}. The point of view we adopt here consists in working with the different notions of {\em orbifolds} that have appeared in the literature: the $V$-manifolds of Satake, the orbifolds of Thurston, the algebraic stacks of Grothendieck, Deligne and Mumford, and the geometric orbifolds of Campana.

In particular, as initiated in \cite{CP}, we extend to the orbifold setting the strategy of Bogomolov \cite{Bo77} which uses symmetric differentials to obtain hyperbolicity properties for surfaces which satisfy $c_1^2-c_2>0$. More precisely, we use Kawasaki-To\"en's Riemann-Roch formula on stacks (\cite{Ka}, \cite{To}) to produce orbifold symmetric differentials.

Then, in the case of smooth (geometric) orbifolds (i.e. $X$ is smooth and $\lceil \Delta \rceil$, the support of $\Delta$, is a normal crossing divisor), we obtain using moreover McQuillan's techniques \cite{McQ0} as in \cite{Rou08}
\begin{thma}
Let  $(X,\Delta)$ be a smooth projective orbifold surface of general type, i.e $K_{X}+\Delta$ is big, $\Delta=\sum_{i} (1-\frac{1}{m_{i}})C_{i}$. Denote $g_{i}:=g(C_{i})$ the genus of the curve $C_{i}$ and $\overline{c}_{1}, \overline{c}_{2}$ the logarithmic Chern classes of $(X,\lceil \Delta \rceil).$
If
\begin{equation}\label{ineq}
\overline{c}_{1}^2-\overline{c}_{2}-\sum_{i=1}^{n}\frac{1}{m_{i}}(2g_{i}-2+\sum_{j\neq i}C_{i}C_{j}) +\sum_{1\leq i \leq j \leq n}\frac{C_{i}C_{j}}{m_{i}m_{j}}>0,
\end{equation}
then there exists a proper subvariety $Y\subsetneq X$ such that every non-constant entire curve $f: \bC \rightarrow X$ which is an orbifold morphism, i.e ramified over $C_{i}$ with multiplicity at least $m_{i}$, verifies $f(\bC) \subset Y$.
\end{thma}

One advantage of this new approach is that we can generalize it to the singular case, for example when the orbifold surface $(X,\Delta)$ is Kawamata log terminal, following the terminology of the Mori Program (see for example \cite{Mat}). This point of view unifies several former results (e.g. \cite{CG}, \cite{GraPe} and \cite{BoDe}) where people have noticed that singularities can help to prove degeneracy statements on holomorphic maps. The key point here is to realize that singularities help to produce orbifold symmetric differentials on the stack associated to the orbifold.

As applications, we obtain as a first example (compare with \cite{CG} and \cite{GraPe})
\begin{thmb}
Let $C \subset \bP^2$ be a curve of degree $d\geq4$ with $n$ nodes and $c$ cusps. If
$$-d^2-15d+\frac{75}{2}+\frac{1079}{96}c+6n>0,$$
then there exists a curve $D\subset \bP^2$ which contains any non-constant entire curve $f:\bC \to \bP^2\setminus C$.
\end{thmb}

The above numerical conditions should be seen as the equivalent of $c_1^2-c_2>0$ in the orbifold setting.
A second example is the case of nodal surfaces $X\subset \bP^3$ of general type of degree $d$ with $l$ nodes where we recover a result of \cite{BoDe} giving the existence of orbifold symmetric differentials as soon as $l>\frac{8}{3}(d^2-\frac{5}{2}d),$ which is unfortunately not satisfied for $d=5$ where the maximum number of nodes is $31$.

We extend our study to higher order orbifold jet differentials and obtain, towards the existence of an hyperbolic quintic,
\begin{thmc}
Let $X\subset \bP^3$ be a nodal quintic with the maximum number of nodes, $31$. Then every classical orbifold entire curve satisfies an algebraic differential equation of order $3$.
\end{thmc}

The paper is organized as follows. In section 2, we recall the basic facts on orbifold structures. In section 3, we describe orbifold symmetric differentials and orbifold morphisms. Then, in section 4, we recall Kawasaki-To\"en's Riemann-Roch formula on orbifolds. In section 5, we study the smooth case and in section 6, the singular case. In section 7, we give applications to complements of plane curves and nodal surfaces. Finally, in section 8, we give definitions and applications of orbifold jet differentials.

\textbf{Acknowledgements}. The approach using stacks has been suggested to us by Philippe Eyssidieux and Michael McQuillan, so we would like to thank them warmly for their interest in this work. We also thank Fr\'ed\'eric Campana for many interesting discussions.

\section{Orbifolds as pairs}
As in \cite{GK} (or \cite{DM}, \S14) we look at orbifolds as a particular type of log pairs. $(X,\Delta)$ is a log pair if $X$ is a normal algebraic variety (or a normal complex space) and $\Delta=\sum_id_iD_i$ is an effective $\bQ$-divisor where the $D_i$ are distinct, irreducible divisors and $d_i \in \bQ$. 

For orbifolds, we need to consider only pairs $(X,\Delta)$ such that $\Delta$ has the form
$$\Delta=\sum_i \left(1-\frac{1}{m_i}\right) D_i,$$
where the $D_i$ are prime divisors and $m_i \in \bN$. These pairs are called {\em geometric orbifolds} by Campana in \cite{C04} and \cite{C07}.

\begin{definition}
An {\em orbifold chart} on $X$ compatible with $\Delta$ is a Galois covering $\varphi : U \to \phi(U)\subset X$ such that
\begin{enumerate}
\item $U$ is a domain in $\bC^n$ and $\varphi(U)$ is open in $X$,
\item the branch locus of $\varphi$ is $\lceil \Delta \rceil \cap \varphi(U)$,
\item for any $x \in U'':=U\setminus \varphi^{-1}(X_{sing}\cup \Delta_{sing})$ such that $\varphi(x) \in D_i$, the ramification order of $\varphi$ at $x$ verifies $ord_{\varphi}(x)=m_i.$
\end{enumerate}
\end{definition}

\begin{definition}
An orbifold $\mathcal{X}$ is a log pair $(X,\Delta)$ such that $X$ is covered by orbifold charts compatible with $\Delta.$
\end{definition}

\begin{remark}
\begin{enumerate}

\item In the language of stacks, we have a smooth Deligne-Mumford stack $\pi: \mathcal{X} \rightarrow X$, with coarse moduli space $X$

\item Geometric orbifolds $(X,\Delta)$ of Campana \cite{C04} are more general since they are not supposed to be locally uniformizable. We have an injective mapping $\mathcal{X} \to (X,\Delta)$ but most pairs $(X,\Delta)$ are not in the image.

\item One can also take infinite $m_i$. The components with $m_i=\infty$ are added in the quasiprojective case to compactify $X$.
\end{enumerate}
\end{remark}

\begin{example}\label{nc}
Let $X$ be a complex manifold and $\Delta=\sum_{i} (1-\frac{1}{m_{i}})D_{i}$ with a support $\lceil \Delta \rceil$ which is a normal crossing divisor, i.e. for any point $x\in X$ there is a holomorphic coordinate system $(V,z_1,\dots,z_n)$ such that $\Delta$ has equation $$z_{1}^{(1-\frac{1}{m_{1}})}\dots z_{n}^{(1-\frac{1}{m_{n}})}=0.$$ Then $(X,\Delta)$ is an orbifold. Indeed, fix a coordinate system as above. Set
$$\varphi: U \to V, \text{   } \phi(x_1,\dots,x_n)=(x_1^{m_1},\dots,x_n^{m_n}).$$
Then $(U,\phi)$ is an orbifold chart on $X$ compatible with $\Delta$.

Equivalently, we have a smooth Deligne-Mumford stack $\pi: \mathcal{X} \rightarrow X$, with coarse moduli space $X$, described locally as follows. For every open polydisk $\bD \subset X$ with local coordinates $(z_{1},\dots,z_{n})$, such that $\Delta$ has equation $z_{1}^{(1-\frac{1}{m_{1}})}\dots z_{n}^{(1-\frac{1}{m_{n}})}=0$ we have: $$\mathcal{X}\times_{X} {\bD}= [{\bD'}/G],$$
where $G=\prod_{j=1}^{n} \mathbb{Z} / m_{j} \mathbb{Z}$ acts on the polydisk $\bD'$ by $(\zeta_{1},...,\zeta_{n}).(y_{1},\dots,y_{n})=(\zeta_{1}y_{1},...,\zeta_{n}y_{n})$ where we identify $ \mathbb{Z} / m_{j} \mathbb{Z}$ and the group of $m_{j}$-th root of unity.
\end{example}

\begin{remark}
The orbifolds of the previous example are said to be {\em smooth} (see \cite{C07}).
\end{remark}

More examples of orbifolds are obtained looking, in the case of surfaces, at different classes of singularities that naturally appear in the logarithmic Mori program (see for example \cite{Mat}).
\begin{definition}
Let $(X,\Delta)$, $\Delta=\sum_i\left(1-\frac{1}{m_i}\right)C_i$, be a pair where $X$ is a normal surface and $K_X+\Delta$ is $\mathbb{Q}$-Cartier. Let $\pi: \tilde{X} \to X$ be a resolution of the singularities of $(X,\Delta)$, so that the exceptional divisors, $E_i$ and the components of $\tilde{\Delta}$, the strict transform of $\Delta$, have normal crossings and
$$K_{\tilde{X}}+\tilde{\Delta}+\sum_iE_i=\pi^*(K_X+\Delta)+\sum_ia_iE_i.$$
\begin{enumerate}
\item We say that $(X,\Delta)$ is \emph{log canonical} if $a_i\geq 0$ for every exceptional curve $E_i$.
\item We say that $(X,\Delta)$ is \emph{klt} (Kawamata log terminal) if $m_i<\infty$ and $a_i > 0$ for every exceptional curve $E_i$.
\end{enumerate}
\end{definition}

The classification of log canonical singularities can be found in \cite{Ko}. The important point is that if $(X,\Delta,x)$ is a germ of a klt surface, then it is analytically equivalent to the quotient of $\bC^2$ by a finite subgroup $G\subset GL(2,\bC)$ and the $C_i$ correspond to the components of the branch locus of the quotient map $p=\bC^2\to \bC^2/G$. The ramification index over a component $C_i$ is equal to $m_i$. 

So, we obtain new examples of orbifold surfaces (\cite{Ulud}, \cite{Ko}, \cite{Meg}):

\begin{example}\label{lc}
Let $(X,\Delta)$, $\Delta=\sum_i\left(1-\frac{1}{m_i}\right)C_i$, be a log canonical pair with $X$ a surface, where all points which are not klt lie on $\lfloor \Delta \rfloor$, then  $(X,\Delta)$ is an orbifold.
\end{example}

\section{Orbifold morphisms and orbifold symmetric differentials}\label{stack}
\subsection{The smooth case}

Let $(X,\Delta)$ be a smooth orbifold, i.e $X$ is a smooth complex manifold and $\Delta=\sum_{i} (1-\frac{1}{m_{i}})Z_{i}$ has a support $\lceil \Delta \rceil$ which is a normal crossing divisor.

Complex hyperbolic aspects of one-dimensional orbifolds have been studied in \cite{CW} and, in \cite{Rou08}, we have started the investigation of the higher dimensional case.

We want to study orbifold holomorphic maps $f:\bC \to (X,\Delta)$. They are defined following \cite{C07} as
\begin{definition} 
Let $(X,\Delta)$ be a smooth orbifold with $\Delta=\sum_{i} (1-\frac{1}{m_{i}})Z_{i}$, $\bD=\{z \in \bC / |z|<1\}$ the unit disk and $h$ a holomorphic map from $\bD$ to $X$.
\begin{enumerate}
\item $h$ is a (non-classical) orbifold morphism from $\bD$ to $(X,\Delta)$ if $h(\bD) \nsubseteq supp(\Delta)$ and $mult_{x}(h^*Z_{i})\geqslant m_{i}$ for all $i$ and $x\in \bD$ with $h(x)\in supp(Z_{i})$. If $m_{i}=\infty$ we require $h(\bD)\cap Z_{i} = \emptyset$.

\item $h$ is a classical orbifold morphism from $\bD$ to $(X,\Delta)$ if $h(\bD) \nsubseteq supp(\Delta)$ and $mult_{x}(h^*Z_{i})$ is a multiple of $m_{i}$ for all $i$ and $x\in \bD$ with $h(x)\in supp(Z_{i})$. If $m_{i}=\infty$ we require $h(\bD)\cap Z_{i} = \emptyset$.

\end{enumerate}
\end{definition}

In the compact or logarithmic setting, symmetric differentials turned out to be key objects for such a study (see for example \cite{Bo77}, \cite{McQ0}).
Let $(x_{1},\dots,x_{n})$ be local coordinates such that $\Delta$ has equation $$x_{1}^{(1-\frac{1}{m_{1}})}\dots x_{n}^{(1-\frac{1}{m_{n}})}=0.$$
Let us recall the definition of sheaves of differential forms on orbifolds (see \cite{C07} for details).
\begin{definition}
For $N$ a positive integer, $S^{N}\Omega_{(X,\Delta)}$ is the locally free subsheaf of $S^{N}\Omega_{X}(log\lceil \Delta \rceil)$ generated by the elements $$x_{1}^{\lceil{\frac{\alpha_{1}}{m_{1}}}\rceil}\dots x_{n}^{\lceil{\frac{\alpha_{n}}{m_{n}}}\rceil}\left(\frac{dx_{1}}{x_{1}}\right)^{\alpha_{1}}\dots\left(\frac{dx_{n}}{x_{n}}\right)^{\alpha_{n}},$$
such that $\sum \alpha_{i}=N$, where $\lceil k \rceil$ denotes the round up of $k$.
\end{definition}

\begin{remark}
One motivation for this definition is that these orbifold symmetric differentials act on orbifold morphisms, i.e. for every orbifold morphism $h: \bD \to (X,\Delta)$, $h^*S^{N}\Omega_{(X,\Delta)}\subset S^N\Omega_\bD$. Moreover, this property characterizes orbifold morphisms \cite{C07}.
\end{remark}

Now we consider the smooth Deligne-Mumford stack $\pi: \mathcal{X} \rightarrow X$, with coarse moduli space $X$, as in example \ref{nc} above, and remark that
\begin{proposition}\label{symstack}
$$\pi_{*}S^{N}\Omega_{\mathcal{X}}=S^{N}\Omega_{(X,\Delta)}.$$
\end{proposition}
\begin{proof}
Note first that $\pi^{*}S^{N}\Omega_{(X,\Delta)}\subset S^{N}\Omega_{\mathcal{X}}$. Then the isomorphism between $\pi_{*}S^{N}\Omega_{\mathcal{X}}$ and $S^{N}\Omega_{(X,\Delta)}$ can be verified locally. Take 
$$y_{1}^{\alpha_{1}}\dots y_{n}^{\alpha_{1}}(dy_{1})^{\beta_{1}}\dots (dy_{n})^{\beta_{n}},$$
where $\alpha_{i} \geq 0$ for all $i$.
Then the assertion is equivalent to the fact that the preceding form is invariant under $G=\prod_{j=1}^{n} \mathbb{Z} / m_{j} \mathbb{Z}$ if and only if $m_{i}|\alpha_{i}+\beta_{i}$ and $\frac{\alpha_{i}+\beta_{i}}{m_{i}}\geq \lceil{\frac{\beta_{i}}{m_{i}}}\rceil$. This follows immediately from the definition of the action.
\end{proof}

Moreover $R^q\pi_{*}S^{N}\Omega_{\mathcal{X}}=0$ for $q>0$ (see \cite{MO}).
Therefore $$\chi(\mathcal{X},S^{N}\Omega_{\mathcal{X}})=\chi(X,S^{N}\Omega_{(X,\Delta)}).$$

Towards the existence of global sections of $S^{N}\Omega_{(X,\Delta)}$, we will compute $\chi(\mathcal{X},S^{N}\Omega_{\mathcal{X}})$ using Kawasaki-To\"en's Riemann-Roch formula (\cite{Ka}, \cite{To}) in the case of orbifold surfaces.

\subsection{The general case}
Following the philosophy of the previous section, we can extend the above definitions to any orbifold $(X,\Delta)$. We denote $\pi: \mathcal{X} \to X$ the Deligne-Mumford stack associated to $(X,\Delta)$.

\begin{definition}Let $(X,\Delta)$ be an orbifold. For $N$ a positive integers, the sheaf $S^{N}\Omega_{(X,\Delta)}$ of orbifold symmetric differentials is defined to be
$$S^{N}\Omega_{(X,\Delta)}:=\pi_{*}S^{N}\Omega_{\mathcal{X}}.$$
\end{definition}

\begin{definition}
Let $(X,\Delta)$ be an orbifold.
\begin{enumerate}
\item A  holomorphic map $f:\bD \to (X,\Delta)$ is a classical orbifold map if it admits a lift $\tilde{f}:\bD \to \mathcal{X}$ to the Deligne-Mumford stack associated to $(X,\Delta)$.
\item A  holomorphic map $f:\bD \to (X,\Delta)$ is a (non-classical) orbifold map if $f^*(S^{N}\Omega_{(X,\Delta)})\subset S^N\Omega_{\bD}$ for all positive integers $N$.
\end{enumerate}
\end{definition}

\section{Kawasaki-To\"en's Riemann-Roch formula}
We want to prove the existence of global orbifold symmetric differentials. For this, we shall apply Riemann-Roch on orbifolds. This was done in \cite{Ka} and generalized in \cite{To} to more general Deligne-Mumford stacks. We shall follow the latter approach.

Let us recall briefly To\"en's Riemann-Roch formula on Deligne-Mumford stacks following \cite{To} and \cite{Ts}. There is an \'etale cohomology theory on stacks which enables to define Chern classes and Todd classes (see \cite{To}). For this theory, if $p:\mathcal{X} \rightarrow X$ is the projection from a stack to its moduli space, there is an isomorphism $p_{*}:A(\mathcal{X}) \simeq A(X)$. The key point of To\"en's formula is that the correct cohomology to work with is that of the inertia stack, defined below. The components of the inertia stack give correction terms to the standard Riemann-Roch formula.

\begin{definition}
Let $\mathcal{X}$ be a Deligne-Mumford stack. The inertia stack $I\mathcal{X}$ associated to $\mathcal{X}$ is defined to be the fiber product
$$I\mathcal{X}:=\mathcal{X}\times_{\mathcal{X}\times\mathcal{X}}\mathcal{X}.$$
\end{definition}

Locally one may describe $I\mathcal{X}$ as follows. If $X$ is a variety, $H$ a finite group acting   on $X$ and $F=[X/H]$ the quotient stack then
$$I_{F}\simeq \coprod _{h\in c(H)} [X^h/Z_{h}],$$
where $X^h \subset X$ is the fixed locus by $h$, $Z_{h}$ the centralizer of $h$ in $H$ and $c(H)$ the conjugacy classes of $H$.

There is a natural projection $q: I\mathcal{X} \rightarrow \mathcal{X}$. We write
$$I\mathcal{X}= \coprod _{i\in I} \mathcal{X}_i$$
for the decomposition of $I\mathcal{X}$ into a disjoint union of connected components. There is a distinguished component $\mathcal{X}_0$, corresponding to $h=1$, which is isomorphic to $\mathcal{X}.$

If $F$ is a vector bundle, then $q^*F$ decomposes into a direct sum of eigen-subbundles
$$\bigoplus_{\zeta \in \mu_{
\infty}} F^{(\zeta)},$$
where $\mu_{\infty}$ is the group of roots of unity. For such a decomposition one defines a map $\rho : K^0(I\mathcal{X}) \rightarrow K^0(I\mathcal{X})$ by
$$\rho(\bigoplus_{\zeta \in \mu_{\infty}} F^{(\zeta)}):=\sum_{\zeta} \zeta F^{(\zeta)}.$$

Then one defines
\begin{definition}\cite{To}
Define $\widetilde{ch}: K^0(\mathcal{X}) \rightarrow H^*(I\mathcal{X})$ to be the composite
$$K^0(\mathcal{X}) \xrightarrow{q^*F} K^0(I\mathcal{X}) \xrightarrow{\rho}K^0(I\mathcal{X}) \xrightarrow{ch} H^*(I\mathcal{X}),$$
where $ch$ is the usual Chern character.
\end{definition}

\begin{definition}\cite{To}
Let $E$ be a vector bundle on $\mathcal{X}$ and $q^*E$ decomposed into a direct sum $(q^*E)^{inv}\oplus (q^*E)^{mov}$ where $(q^*E)^{inv}$ is the eigenbundle with eigenvalue $1$ and $(q^*E)^{mov}$ is the direct sum with eigenvalues not equal to $1$. Then define
$\widetilde{Td}: K^0(\mathcal{X}) \rightarrow H^*(I\mathcal{X})$ by
$$\widetilde{Td}(E):=\frac{Td((q^*E)^{inv}}{ch(\rho \circ \lambda_{-1}(((q^*E)^{mov})^{*}))},$$
where $\lambda_{-1}$ is defined by $\lambda_{-1}(V):=\sum_{a\geq 0}(-1)^a \bigwedge^aV$ for a vector bundle $V$.
\end{definition}

Then To\"en's Riemann-Roch formula gives
\begin{theorem}\label{RR}\cite{To}
Let $\mathcal{X}$ be a Deligne-Mumford stack with quasi-projective coarse moduli space and which has the resolution property (i.e every coherent sheaf is a quotient of a vector bundle). Let $E$ be a coherent sheaf on $\mathcal{X}$ then
$$\chi(\mathcal{X},E)=\int_{\mathcal{X}}\widetilde{ch}(E)\widetilde{Td}(T_{\mathcal{X}}).$$
\end{theorem}

\section{Applications to smooth orbifold surfaces of general type}
Let us apply To\"en's formula in our situation.  Our observation here is that we are only interested in asymptotic Riemann-Roch, therefore the only contribution that we have to take into account is the one coming from the component of the inertia stack of maximal dimension i.e the stack itself. In other words the \'etale cohomology is enough to deal with asymptotic Riemann-Roch.

\begin{theorem}\label{Bogorb}
Let  $(X,\Delta)$ be a smooth projective orbifold surface, $\Delta=\sum_{i} (1-\frac{1}{m_{i}})C_{i}$. Then
$$\chi(\mathcal{X},S^{N}\Omega_{\mathcal{X}})=\frac{N^3}{6}(c_{1}^2-c_{2})+O(N^2),$$
 where $\mathcal{X}$is the stack associated to $X$ described in example \ref{nc}, $c_{1}$ and $c_{2}$ are the  \'etale orbifold Chern classes of $\mathcal{X}$.
 \end{theorem}
 
 \begin{proof}
 Let $Y:=\bP(\Omega_{\mathcal{X}})$ and $L:=\mathcal{O}_{Y}(1)$, then $\chi(\mathcal{X},S^{N}\Omega_{\mathcal{X}})=\chi(Y,L^{\otimes N}).$ The inertia stack $p: IY\rightarrow Y$ is decomposed in connected components $$IY=Y\coprod _{i=1}^{n} Y_i \coprod _{1\leq i<j \leq n} \coprod _{k=1}^{C_{i}C_{j}} Y_{i,j,k},$$
 where $Y_i$ lies over $C_{i}$ and  $Y_{i,j,k}$ over $C_{i}\cap C_{j}$.
 Corresponding to this decomposition we have
 $$\widetilde{ch}(L^{\otimes N})=ch(L^{\otimes N})\bigoplus_{i} \zeta_{i}^{N} ch(L_{i}^{\otimes N}) \bigoplus_{i,j,k} \zeta_{i,j,k}^{N} ch(L_{i,j,k}^{\otimes N}),$$
 where $L_{i}$ and $L_{i,j,k}$ denotes the restrictions of $p^*L$ to $Y_i$ and $Y_{i,j,k}.$
 We apply To\"en's Riemann-Roch formula and obtain
 $$\chi(Y,L^{\otimes N})=\frac{c_{1}(L)^3}{6}N^3+O(N^2),$$
 since the terms coming from the $L_{i}$ and $L_{i,j,k}$ are all $O(N^2)$ because of the dimension.
 This concludes the proof by the classical formula relating $c_{1}(L)$, $c_{1}$ and $c_{2}$.
 \end{proof}
 
 Now we compute the \'etale orbifold Chern classes.
 The following "Gauss-Bonnet" formula will be useful
 \begin{proposition}\label{GB}\cite{To}
 Let $\mathcal{X}$ be a Deligne-Mumford stack of dimension $n$ with the same hypotheses as in theorem \ref{RR}. Let $\{M_{i}\}_{i}$ be a stratification of its coarse moduli space such that the order of ramification of $\mathcal{X}$ is constant on each $M_{i}$ equal to $m_{i}$. Then
 $$\int_{\mathcal{X}}c_{n}=\sum_{i}\frac{\chi(M_{i})}{m_{i}}.$$
 \end{proposition}
 
 Then we can compute explicitely the orbifold Chern classes
 \begin{proposition}
Let  $(X,\Delta)$ be a smooth projective orbifold surface, $\Delta=\sum_{i} (1-\frac{1}{m_{i}})C_{i}$. Denote $g_{i}:=g(C_{i})$ the genus of the curve $C_{i}$ and $\overline{c}_{1}, \overline{c}_{2}$ the logarithmic Chern classes of $(X,\lceil \Delta \rceil).$ Then the \'etale orbifold Chern classes $c_{1}, c_{2}$ of the stack $\mathcal{X}$ associated to $(X,\Delta)$ verify
 \begin{eqnarray*}
c_{1}^2=\overline{c}_{1}^2-2\sum_{i=1}^{n}\frac{1}{m_{i}}(2g_{i}-2)+\sum_{i=1}^{n}\frac{C_{i}^2}{m_{i}^2}+2\sum_{1\leq i<j \leq n} \frac{C_{i}C_{j}}{m_{i}m_{j}}-2\sum_{j=1}^{n}\sum_{i=1,i\neq j}^{n}\frac{C_{i}C_{j}}{m_{j}},\\
c_{2}=\overline{c}_{2}-\sum_{i=1}^{n}\frac{1}{m_{i}}(2g_{i}-2)-\sum_{j=1}^{n}\sum_{i=1,i\neq j}^{n}\frac{C_{i}C_{j}}{m_{j}}+\sum_{1\leq i<j \leq n}\frac{C_{i}C_{j}}{m_{i}m_{j}}.
 \end{eqnarray*}
 \end{proposition}
 
 \begin{proof}
 We have $$K_{\mathcal{X}}=\pi^*(K_{X}+\Delta),$$
 therefore
 \begin{eqnarray*}
 c_{1}^2=\left(K_{X}+\Delta \right)^2=\left((K_{X}+\sum_{i=1}^n C_{i})-\sum_{i=1}^n\frac{1}{m_{i}}C_{i}\right)^2\\
=\overline{c}_{1}^2-2\sum_{j=1}^{n} \frac{1}{m_{j}}(K_{X}+\sum_{i=1}^{n} C_{i})C_{j}+\left(\sum_{i=1}^n \frac{1}{m_{i}} C_{i} \right)^2.
\end{eqnarray*}
We have $K_{X}C_{j}=(2g_{j}-2)-C_{j}^2,$ therefore we obtain
 \begin{eqnarray*}
 c_{1}^2=\overline{c}_{1}^2-2\sum_{i=1}^{n}\frac{1}{m_{i}}(2g_{i}-2)+2\sum_{i=1}^{n}\frac{C_{i}^2}{m_{i}}-2\sum_{j=1}^{n}\sum_{i=1}^{n}\frac{C_{i}C_{j}}{m_{j}}+\sum_{i=1}^{n}\frac{C_{i}^2}{m_{i}^2}+2\sum_{1\leq i<j \leq n} \frac{C_{i}C_{j}}{m_{i}m_{j}}\\
 =\overline{c}_{1}^2-2\sum_{i=1}^{n}\frac{1}{m_{i}}(2g_{i}-2)+\sum_{i=1}^{n}\frac{C_{i}^2}{m_{i}^2}+2\sum_{1\leq i<j \leq n} \frac{C_{i}C_{j}}{m_{i}m_{j}}-2\sum_{j=1}^{n}\sum_{i=1,i\neq j}^{n}\frac{C_{i}C_{j}}{m_{j}}.
 \end{eqnarray*}
 For $c_{2}$ we use the previous proposition \ref{GB} which gives
 \begin{eqnarray*}
c_{2}=\chi(X)-\chi\left(\bigcup_{i=1}^{n} C_{i}\right) + \sum_{i=1}^{n}\frac{1}{m_{i}}\chi\left(C_{i} \setminus \bigcup_{j=1,j \neq i}^n C_{i}\cap C_{j}\right)+\sum_{1\leq i<j\leq n}\frac{1}{m_{i}m_{j}}\chi(C_{i} \cap C_{j})\\
=\overline{c}_{2}-\sum_{i=1}^{n}\frac{1}{m_{i}}(2g_{i}-2)-\sum_{j=1}^{n}\sum_{i=1,i\neq j}^{n}\frac{C_{i}C_{j}}{m_{j}}+\sum_{1\leq i<j \leq n}\frac{C_{i}C_{j}}{m_{i}m_{j}}.
  \end{eqnarray*}
 \end{proof}
 
 As a corollary we obtain
 \begin{corollary}
 Let  $(X,\Delta)$ be a smooth projective orbifold surface of general type, i.e $K_{X}+\Delta$ is big, $\Delta=\sum_{i} (1-\frac{1}{m_{i}})C_{i}$ and $A$ an ample line bundle on $X$. Denote $g_{i}:=g(C_{i})$ the genus of the curve $C_{i}$.
If
$$\overline{c}_{1}^2-\overline{c}_{2}-\sum_{i=1}^{n}\frac{1}{m_{i}}(2g_{i}-2+\sum_{j\neq i}C_{i}C_{j}) +\sum_{1\leq i \leq j \leq n}\frac{C_{i}C_{j}}{m_{i}m_{j}}>0,$$
then $H^0(X,S^{N}\Omega_{(X,\Delta)}\otimes A^{-1})\neq 0$ for $N$ large enough.
 \end{corollary}
 \begin{proof}
 With the hypotheses, theorem \ref{Bogorb} gives that 
 $$h^0(\mathcal{X},S^{N}\Omega_{\mathcal{X}})+h^2(\mathcal{X},S^{N}\Omega_{\mathcal{X}})\geq cN^3$$ for some suitable positive constant $c$ and all sufficiently large integers $N$. By Serre duality $h^2(\mathcal{X},S^{N}\Omega_{\mathcal{X}})=h^0(\mathcal{X},S^{N}\Omega_{\mathcal{X}}\otimes K_{\mathcal{X}}^{1-N}),$ and as $K_{\mathcal{X}}=\pi^*(K_{X}+\Delta),$ we obtain an injection $h^2(\mathcal{X},S^{N}\Omega_{\mathcal{X}}) \hookrightarrow h^0(\mathcal{X},S^{N}\Omega_{\mathcal{X}})$. Therefore $h^0(\mathcal{X},S^{N}\Omega_{\mathcal{X}}) \geq \frac{c}{2}N^3.$ Since $h^0(\mathcal{X},S^{N}\Omega_{\mathcal{X}})=h^0(X,S^{N}\Omega_{(X,\Delta)})$, this concludes the proof.
 \end{proof}
 
Now let us recall that in \cite{Rou08} we have obtained
\begin{theorem}\label{degen}
Let  $(X,\Delta)$ be a smooth projective orbifold surface of general type, $A$ an ample line bundle on $X$ such that $H^0(X,S^N\Omega_{(X,\Delta)}\otimes A^{-1})\neq 0$ for $N$ large enough. Then there exists a proper subvariety $Y\subsetneq X$ such that every non-constant entire curve $f: \bC \rightarrow X$ which is an orbifold morphism verifies $f(\bC) \subset Y$.
\end{theorem}

As an immediate corollary we obtain the theorem announced
\begin{thma}
Let  $(X,\Delta)$ be a smooth projective orbifold surface of general type, $\Delta=\sum_{i} (1-\frac{1}{m_{i}})C_{i}$. Denote $g_{i}:=g(C_{i})$ the genus of the curve $C_{i}$ and $\overline{c}_{1}, \overline{c}_{2}$ the logarithmic Chern classes of $(X,\lceil \Delta \rceil).$
If
\begin{equation*}\tag{\ref{ineq}}
\overline{c}_{1}^2-\overline{c}_{2}-\sum_{i=1}^{n}\frac{1}{m_{i}}(2g_{i}-2+\sum_{j\neq i}C_{i}C_{j}) +\sum_{1\leq i \leq j \leq n}\frac{C_{i}C_{j}}{m_{i}m_{j}}>0,
\end{equation*}
then there exists a proper subvariety $Y\subsetneq X$ such that every non-constant entire curve $f: \bC \rightarrow X$ which is an orbifold morphism verifies $f(\bC) \subset Y$.
\end{thma}

\begin{remark}
This result generalizes and implies as a particular case the corresponding theorem of \cite{Rou08} where the hypotheses were much stronger. Indeed it was needed that $g_i\geq 2$, $h^0(C_{i},\mathcal{O}_{C_{i}}(C_{i})) \neq 0$ for all $i$ and that the logarithmic Chern classes of $(X, \lceil \Delta \rceil)$ had to verify
\begin{equation*}
\overline{c_{1}}^2-\overline{c_{2}}-\sum_{i=1}^{n}\frac{1}{m_{i}}(2g_{i}-2+\sum_{j \neq i}C_{i}C_{j})>0.
\end{equation*}
\end{remark}

\begin{remark}
One can write the previous inequality \ref{ineq} in terms of Chern classes $d_1, d_2$ of $X$ and quantities involving only $K_X$ and $\Delta$. It becomes
$$d_1^2-d_2+2K_X\Delta+\Delta^2+\chi(\Delta)>0,$$
where $\chi(\Delta)=\chi\left(\sum_{i} (1-\frac{1}{m_{i}})C_{i}\right):=\sum_{i} (1-\frac{1}{m_{i}})\chi(C_{i})-\sum_{i<j} (1-\frac{1}{m_{i}})(1-\frac{1}{m_{j}})C_{i}C_j.$
\end{remark}

As an application we obtain the following theorem
\begin{theorem}
Let $C_{i}, 1\leq i \leq 2$, be two smooth curves in $X=\bP_{2}$ of degree $d_{i}\geq 4$ with normal crossings. Let $\Delta=(1-\frac{1}{m_{1}})C_{1}+(1-\frac{1}{m_{2}})C_{2}$, and $d=d_{1}+d_{2}$. If 
\begin{equation}
\label{eq1}
deg(\Delta)^2-deg(\Delta)(d+3)+d_1d_2\left(1-\frac{1}{m_1m_2}\right)+6>0
\end{equation}
then there exists a curve $D \subset X$ which contains every orbifold entire curve $f: \bC \rightarrow (X,\Delta)$.
\end{theorem}

\begin{proof}
First we verify that condition \ref{ineq} is satisfied. We compute everything in terms of the degrees $d_{1}\leq d_{2}$
\begin{eqnarray*}
\overline{c_{1}}^2-\overline{c_{2}}-\frac{1}{m_{1}}(2g_{1}-2+d_{1}d_{2})-\frac{1}{m_{2}}(2g_{2}-2+d_{1}d_{2})+\frac{d_1^2}{m_1^2}+\frac{d_2^2}{m_2^2}+\frac{d_1d_2}{m_1m_2}=\\
deg(\Delta)^2-deg(\Delta)(d+3)+d_1d_2\left(1-\frac{1}{m_1m_2}\right)+6.
\end{eqnarray*}
So, if condition \ref{eq1} is satisfied, we can apply theorem A and obtain the algebraic degeneracy of $f$.
\end{proof}

\begin{example}
Let $C_{i}, 1\leq i \leq 2$ be two smooth curves in $\bP_{2}$ of degree $5$ with normal crossings. Let $\Delta=(1-\frac{1}{69})C_{1}+(1-\frac{1}{69})C_{2}$. Then there exists a curve $D \subset \bP^2$ which contains every orbifold entire curve $f: \bC \rightarrow (X,\Delta)$. If the curves $C_{i}$ are very generic, then $(\bP_{2},\Delta)$ is hyperbolic (see \cite{Rou08}).
\end{example}

\section{The singular case}
We can apply the above ideas to the second class of examples, namely klt surfaces $(X,\Delta)$.

In the classical case one obtains as an immediate consequence of \cite{McQ}, \cite{McQ2}

\begin{theorem}\label{class}
Let $(X,\Delta)$ be a projective klt orbifold surface of general type and $\pi:\mathcal{X}\to X$ its associated Deligne-Mumford stack. If $$c_1^2(\mathcal{X})-c_2(\mathcal{X})>0,$$ then there exists a proper subvariety $Y\subsetneq X$ such that any non-constant classical orbifold entire curve $f:\bC \to (X,\Delta)$ is contained in $Y$.
\end{theorem}
\begin{proof}
By definition $f:\bC \to (X,\Delta)$ lifts to $\tilde{f}:\bC \to \mathcal{X}$. Moreover $c_1^2(\mathcal{X})-c_2(\mathcal{X})>0$ implies that $H^0(\mathcal{X},S^N\Omega_{\mathcal{X}}\otimes \pi^*A^{-1})\neq 0$ for $N$ large enough where $A$ is an ample line bundle on $X$. Then there exists a proper sub-stack $Z$ of $\bP(T_{\mathcal{X}})$ which contains the image of the derivative of $\tilde{f}.$ The main theorem of \cite{McQ} then implies that $\tilde{f}$ factors through a sub-stack $Z'$ of $\mathcal{X}$.
\end{proof}

In the non-classical case we can prove

\begin{theorem}\label{nonclass}
Let $(X,\Delta)$ be a projective klt orbifold surface of general type, $\pi:\mathcal{X}\to X$ its associated Deligne-Mumford stack and $Z$ the subset of $X$ consisting of $Sing(X)$ and the locus in $X\setminus Sing(X)$ where $\lceil \Delta \rceil$ is not a divisor with only normal crossings. If $$c_1^2(\mathcal{X})-c_2(\mathcal{X})>0,$$
then there exists a proper subvariety $Y\subsetneq X$ such that any non-constant orbifold entire curve $f:\bC \to (X,\Delta)$ with the property that $f(\bC)\cap Z=\emptyset$ is contained in $Y$.
\end{theorem}
\begin{proof}
$c_1^2(\mathcal{X})-c_2(\mathcal{X})>0$ implies that there exists $\omega \in H^0(X,S^N\Omega_{(X,\Delta)}\otimes A^{-1})\neq 0$ for $N$ large enough where $A$ is an ample line bundle on $X$. Let $p: \tilde{X} \to X$ be a resolution of the singularities of $(X,\Delta)$, so that the exceptional divisors, $E_i$ and the components of $\tilde{\Delta}$, the strict transform of $\Delta$, have normal crossings and
$$K_{\tilde{X}}+\tilde{\Delta}+\sum_iE_i=p^*(K_X+\Delta)+\sum_ia_iE_i.$$
Let $\tilde{f}: \bC \to \tilde{X}$ be the lifting of $f$. Then $\tilde{f}$ is an orbifold map into $(\tilde{X},\tilde{\Delta}+\sum_iE_i)$ since $f(\bC)\cap Z=\emptyset$ which implies $\tilde{f}(\bC)\cap (\cup_iE_i)=\emptyset$. But, since $(X,\Delta)$ is klt, $(\tilde{X},\tilde{\Delta}+\sum_iE_i)$ is of general type. Moreover $p^*\omega \in H^0(\tilde{X},S^N\Omega_{(\tilde{X},\tilde{\Delta}+\sum_iE_i)}\otimes p^*A^{-1}).$ To finish the proof we just have to apply theorem \ref{degen} to the smooth orbifold of general type $(\tilde{X},\tilde{\Delta}+\sum_iE_i)$.
\end{proof}

In fact, the proof shows that we can do better, namely we can shrink the locus $Z$. Indeed, write the ramification formula
$$K_{\tilde{X}}+\tilde{\Delta}=p^*(K_X+\Delta)+\sum a(E; X, \Delta)E,$$
where $a(E; X, \Delta)$, which is independent of $p$ (see \cite{Mat}), is called the discrepancy of $(X,\Delta)$ at $E$ and $p(E)$ is called the center of $E$ on $X$ denoted by $Center_{X}(E).$ Then the previous proof immediately generalizes as
\begin{theorem}\label{nonclass2}
Let $(X,\Delta)$ be a projective klt orbifold surface of general type, $\pi:\mathcal{X}\to X$ its associated Deligne-Mumford stack, $Z$ the subset of $X$ consisting of $Sing(X)$ and the locus in $X\setminus Sing(X)$ where $\lceil \Delta \rceil$ is not a divisor with only normal crossings and $Z'$ the non-canonical locus, i.e $Z'=\{Center_{X}(E) / a(E; X, \Delta)<0\}$. If $$c_1^2(\mathcal{X})-c_2(\mathcal{X})>0,$$
then there exists a proper subvariety $Y\subsetneq X$ such that any non-constant orbifold entire curve $f:\bC \to (X,\Delta)$ with the property that $f(\bC)\cap Z \cap Z'=\emptyset$ is contained in $Y$.
\end{theorem}

\begin{question}
Can we shrink the "bad" locus so that it becomes empty?
\end{question}

\begin{remark}
As seen in example \ref{lc}, one can generalize slightly the preceding result to the case of log cano\-nical orbifold surface $(X,\Delta)$ where all points which are not klt lie on $\lfloor \Delta \rfloor.$
\end{remark}

\section{Applications to singular orbifold surfaces of general type}
\subsection{Complements of plane curves}
Let us consider a curve $C\subset \bP^2$. We can apply the above results to obtain examples where any holomorphic map $f:\bC \to \bP^2\setminus C$ is contained in a curve $D\subset \bP^2$. Such kind of results have been obtained in \cite{CG} and \cite{GraPe} by different methods. The approach used here shows that "order $1$" techniques, i.e. symmetric differentials, can still be useful in this situation contrary to the smooth case.

Let us illustrate this in the case of a curve $C\subset \bP^2$ with ordinary double points and cusps as singularities. The orbifold $(\bP^2,\alpha C)$ is klt for $\alpha<\frac{5}{6}$ (see for example \cite{Laz}) so the orbifold $(\bP^2,\left(1-\frac{1}{5}\right)C)$ is klt and applying theorem \ref{nonclass} (or theorem \ref{class}), we obtain as announced

\begin{thmb}
Let $C \subset \bP^2$ be a curve of degree $d\geq4$ with $n$ nodes and $c$ cusps. If
$$-d^2-15d+\frac{75}{2}+\frac{1079}{96}c+6n>0,$$
then there exists a curve $D\subset \bP^2$ which contains any non-constant entire curve $f:\bC \to \bP^2\setminus C$.
\end{thmb}

\begin{proof}
Let $\mathcal{X}$ be the stack associated to the klt orbifold $(\bP^2,\left(1-\frac{1}{m}\right)C)$, $m=5$. We just have to compute $c_1^2(\mathcal{X})$ and $c_2(\mathcal{X})$.
We have
$$c_1^2(\mathcal{X})=\left(K_{\bP^2}+\left(1-\frac{1}{5}\right)C\right)^2=\left(-3+\left(1-\frac{1}{5}\right)d\right)^2=9+\frac{16}{25}d^2-\frac{24}{5}d.$$
Now we use proposition \ref{GB} to compute  $c_2(\mathcal{X})$. To do so we need to compute the order of the orbifold fundamental group at singular points of $C$. This can be found in \cite{Ulud} for example. At a node we find that this order is $m^2=25$ and at a cusp, it is equal to
$$
\frac{4}{6}\left(\frac{1}{m}-1+\frac{5}{6}\right)^{-2}=600.$$ 

Therefore we obtain
\[
\aligned
&
c_2(\mathcal{X})=\chi(\bP^2)-\chi(C)+\frac{1}{m}\chi(C\setminus Sing(C))+\frac{1}{25}n+\frac{1}{600}c\\
&=3-\left(1-\frac{1}{m}\right)\chi(C\setminus Sing(C))-\left(1-\frac{1}{25}\right)n-\left(1-\frac{1}{600}\right)c\\
&=3-\frac{4}{5}(\chi(\tilde{C})-2n-c)-\frac{24}{25}n-\frac{599}{600}c\\
&=3-\frac{4}{5}\chi(\tilde{C})+\frac{16}{25}n-\frac{119}{600}c,
\endaligned
\]
where $\tilde{C}$ is the normalization of $C$.
We have
$$g(\tilde{C})=\frac{(d-1)(d-2)}{2}-n-c,$$
therefore
$$\chi(\tilde{C})=2-2g(\tilde{C})=2n+2c-d(d-3).$$
Finally we obtain
$$c_1^2(\mathcal{X})-c_2(\mathcal{X})=\frac{4}{25}(-d^2-15d+\frac{75}{2}+\frac{1079}{96}c+6n).$$
\end{proof}

\begin{remark}
One can compare the previous result with \cite{CG} and \cite{GraPe}, where a stronger property, namely hyperbolicity, is proved but under a numerical condition which can be seen to be much more restrictive than the one obtained here. In particular, all cases of \cite{CG} and \cite{GraPe} must verify $d\geq 9$, which is not the case above.
\end{remark}

\subsection{Singular surfaces}
Let us consider a nodal hypersurface $X\subset \bP^3$, i.e. its singularities are ordinary double points. Hyperbolic properties of such surfaces have been studied in \cite{BoDe}. Here, applying theorem \ref{nonclass2}, we obtain 

\begin{theorem}
Let $X\subset \bP^3$ be a nodal surface of general type of degree $d$ with $l$ nodes. If
$$l>\frac{8}{3}\left(d^2-\frac{5}{2}d\right),$$
then there exists a proper subvariety $Y \subset X$ which contains every non-constant orbifold entire curve $f:\bC \to X.$
\end{theorem}
\begin{proof}
First observe that the singularities are canonical (i.e. $a(E;X)\geq 0$ for $E$ exceptional divisors appearing in a resolution of singularities) so in the notations of theorem \ref{nonclass2} we have $Z\cap Z'=\emptyset$ and so no restrictions on orbifold entire curves.

Let us compute $c_1^2(\mathcal{X})$ and $c_2(\mathcal{X})$ where $\pi: \mathcal{X} \to X$ is the stack associated to $X$. We have
$$K_{\mathcal{X}}=\pi^*K_X$$
and by proposition \ref{GB}
$$c_2(\mathcal{X})=\chi(X\setminus Sing(X))+\frac{l}{2}.$$

Now, consider $p:\tilde{X} \to X$ the minimal resolution of $X$. Then we have
$$K_{\tilde{X}}=p^*K_{X}.$$
So we obtain
$$c_1^2(\mathcal{X})=c_1(\tilde{X})^2$$
and
$$c_2(\mathcal{X})=\chi(\tilde{X} \setminus \cup E_i)+\frac{l}{2}=c_2(\tilde{X})-2l+\frac{l}{2}=c_2(\tilde{X})-\frac{3l}{2},$$
where the $E_i$ are the exceptional curves.
Therefore, we have
$$c_1^2(\mathcal{X})-c_2(\mathcal{X})=c_1(\tilde{X})^2-c_2(\tilde{X})+\frac{3l}{2}.$$
$\tilde{X}$ can be seen as the central fiber of a flat family $X_t \to D$ on the unit disk where the other members are smooth hypersurfaces of degree $d$, so we have
$$c_1(\tilde{X})^2-c_2(\tilde{X})=d(d-4)^2-d(d^2-4d+6)=d(10-4d).$$
And finally
$$c_1^2(\mathcal{X})-c_2(\mathcal{X})=d(10-4d)+\frac{3l}{2}=\frac{3}{2}\left(l-\frac{8}{3}\left(d^2-\frac{5}{2}d\right)\right).$$

\end{proof}

\begin{remark}
One can notice that we obtain exactly the same numerical condition as in \cite{BoDe} and, as it is observed there, there exists surfaces of degree $d\geq 6$ satisfying it but not of degree $5$, since then, the maximum number of nodes is $31$ and $33$ at least is needed. The next section will provide an alternative method to deal with entire curve on such a surface.
\end{remark}

\begin{problem}
Find a singular quintic in $\bP^3$ such that $c_1^2(\mathcal{X})-c_2(\mathcal{X})>0.$
\end{problem}

\section{Orbifold jet differentials}
\subsection{The smooth case}
Recall that if $X$ is a compact complex manifold, in \cite{GG80} Green and Griffiths have introduced the vector bundle of jet differentials of order $k$ and degree $m$, $E_{k,m}^{GG}\Omega_{X}\rightarrow X$ whose fibers are complex valued polynomials $Q(f^{\prime },f^{\prime \prime },\dots,f^{(k)})$ on the fibers of $J_{k}X$ of weight $m$ for the action of  $\mathbb{C}^{\ast }$:
\begin{equation*}
Q(\lambda f^{\prime },\lambda ^{2}f^{\prime \prime },\dots,\lambda
^{k}f^{(k)})=\lambda ^{m}Q(f^{\prime },f^{\prime \prime },\dots,f^{(k)})
\end{equation*}
for all $\lambda \in \mathbb{C}^{\ast }$ and $(f^{\prime },f^{\prime \prime
},\dots,f^{(k)})\in J_{k}X.$

If $(X,D)$ is a smooth logarithmic manifold, i.e $X$ is a complex manifold and $D=\sum_{i} D_{i}$ is a reduced normal crossing divisor, the vector bundle of logarithmic jet differentials of order $k$ and degree $m$, $E_{k,m}^{GG}\Omega_{(X,D)}\rightarrow X$, consists of polynomial operators (satisfying the same weight condition) in the derivatives of order $1, 2,\dots, k$ of $f$ and of the $\log(s_{j}(f))$ where $D_{j}=\{s_{j}=0\}$ locally (see \cite{DL96} for details).

Let $(X,\Delta)$ be a smooth orbifold. Let $(x_{1},\dots,x_{n})$ be local coordinates such that $\Delta$ has equation $$x_{1}^{(1-\frac{1}{m_{1}})}\dots x_{n}^{(1-\frac{1}{m_{n}})}=0.$$
Generalizing the definition of orbifold symmetric differentials, one may define orbifold jet differentials in the following way

\begin{definition}
For $N$ a positive integer, $E^{GG}_{k,N}\Omega_{(X,\Delta)}$ is the locally free subsheaf of $E_{k,N}^{GG}\Omega_{(X,\lceil \Delta \rceil)}$ generated by the elements
$$\prod_{1\leq i \leq n}x_{i}^{\lceil \frac{\alpha_{i,1}+\dots+k\alpha_{i,k}}{m_{i}}\rceil}\left(\frac{dx_{i}}{x_{i}}\right)^{\alpha_{i,1}}\dots\left(\frac{d^kx_{i}}{x_{i}}\right)^{\alpha_{i,k}},$$
such that $|\alpha_{1}|+2|\alpha_{2}|+\dots+k|\alpha_{k}|=N$ where $|\alpha_{i}|=\sum_{j}\alpha_{j,i}$.
\end{definition}

From this definition, it is clear that elements $\omega \in H^0(X,E^{GG}_{k,N}\Omega_{(X,\Delta)}\otimes A^{-1})$ act on orbifold morphisms $f: \bC \to (X,\Delta)$ giving holomorphic sections of $f^*A^{-1}$. More precisely, we have

\begin{theorem}
Let  $(X,\Delta)$ be a smooth compact orbifold, $A$ an ample line bundle on $X$ and $P \in H^0(X,E^{GG}_{k,N}\Omega_{(X,\Delta)}\otimes A^{-1})$. Then for any orbifold morphism $f:\bC \rightarrow (X,\Delta)$
\begin{displaymath}
P(f)\equiv 0.
\end{displaymath}
\end{theorem}

\begin{proof}
The proof goes along the same lines as in the classical setting using the logarithmic derivative lemma (see \cite{Siu}, \cite{W99}, \cite{CP}) which we summarize for the convenience of the reader.
$P(f)$ is a holomorphic section of $f^*A^{-1}$. Suppose it does not vanish identically. Let $\omega=\Theta_{h}(A)$, then by the Poincar\'e-Lelong equation
$$i\partial \overline\partial \log ||P(f)||^2_{h^{-1}} \geq f^*\omega.$$ Therefore
$$T_{f}(r,\omega)\leq \int_{1}^{r}\frac{dt}{t} \int_{|z|<t} i\partial \overline\partial \log ||P(f)||^2_{h^{-1}}$$
and from Jensen formula
$$\int_{0}^{2\pi}\log^{+} ||P(f)||_{h^{-1}}d\theta \geq T_{f}(r,\omega)+\mathcal{O}(1).$$
Finally the logarithmic derivative lemma gives
$$\int_{0}^{2\pi}\log^{+} ||P(f)||_{h^{-1}}d\theta \leq \mathcal{O}(\log(r)+\log(T_{f}(r,\omega)))$$
outside a set of finite Lebesgue measure in $[0,+\infty[$.
This gives a contradiction.
\end{proof}

Another possibility to define orbifold jet differentials, following the philosophy of the preceding sections, is to consider the smooth Deligne-Mumford stack $\pi: \mathcal{X} \rightarrow X$, with coarse moduli space $X$. Then one can define
\begin{definition}\label{symstack}
The sheaf $E_{k,N}\Omega_{(X,\Delta)}$ of classical jet differentials is
$$E_{k,N}\Omega_{(X,\Delta)}:=\pi_{*}E_{k,N}\Omega_{\mathcal{X}}.$$
\end{definition}

The same proof as above gives
\begin{theorem}
Let  $(X,\Delta)$ be a smooth compact orbifold, $A$ an ample line bundle on $X$ and $P \in H^0(X,E_{k,N}\Omega_{(X,\Delta)}\otimes A^{-1})$. Then for any classical orbifold morphism $f:\bC \rightarrow (X,\Delta)$
\begin{displaymath}
P(f)\equiv 0.
\end{displaymath}
\end{theorem}

The situation for higher order orbifold jet differentials turns out to be different from the case of orbifold symmetric differentials. Indeed, in the order $1$ case, the key point is that orbifold symmetric differentials act on classical and non-classical orbifold morphisms. From order $2$, this is not the case anymore as we can see in the following

\begin{example}
Consider the morphism of orbifold $[\bD^{n}/G] \rightarrow \bD^n$ induced by $\pi: (y_{1},\dots,y_{n}) \rightarrow (y_{1}^{m_{1}},\dots,y_{n}^{m_{n}})$. A simple computation gives
$$\omega:=\pi_*(d^2y_i)^{m_i}=y_i\left(\frac{1}{m_i}\left(\frac{1}{m_i}-1\right)\left(\frac{dy_i}{y_i}\right)^2+\frac{1}{m_i}\left(\frac{d^2y_i}{y_i}\right)\right)^{m_i}.$$
Then one can see that if $f:\bD \to \bD^n$ is an orbifold morphism, $\omega(f)$ is not necessary holomorphic except if $mult_x(f^*(y_i=0))\geq 2m_i$ for all $x$ such that $f(x)\in (y_i=0).$
\end{example}

\subsection{The singular case}
Let us study orbifold jet differentials in the case of singular surfaces. 
Let $(X,\Delta)$ be an orbifold and consider the smooth Deligne-Mumford stack $\pi: \mathcal{X} \rightarrow X$, with coarse moduli space $X$. Then one can define, as in the smooth case,
\begin{definition}
The sheaf $E_{k,N}\Omega_{(X,\Delta)}$ of classical jet differentials is
$$E_{k,N}\Omega_{(X,\Delta)}:=\pi_{*}E_{k,N}\Omega_{\mathcal{X}}.$$
\end{definition}

Classical jet differentials act on classical orbifold morphisms and as above we have
\begin{theorem}
Let  $(X,\Delta)$ be an orbifold with $X$ compact, $A$ an ample line bundle on $X$ and $P \in H^0(X,E_{k,N}\Omega_{(X,\Delta)}\otimes A^{-1})$. Then for any classical orbifold morphism $f:\bC \rightarrow (X,\Delta)$
\begin{displaymath}
P(f)\equiv 0.
\end{displaymath}
\end{theorem}

We have seen above a numerical condition for the existence of global orbifold symmetric differentials on nodal surfaces, which unfortunately is not satisfied for nodal quintics. Here we have

\begin{theorem}
Let $X\subset \bP^3$ be a nodal surface of general type of degree $d$ with $l$ nodes.
\begin{enumerate}
\item If $$l>\frac{-4}{15}(d^3-18d^2+41d),$$
then $X$ has global $2$-jet differentials i.e. global sections of $E_{2,N}\Omega_{(X,\Delta)}$. More precisely
$$h^0(X,E_{2,N}\Omega_{X})\geq \left(\frac{15l}{2}+2d^3-36d^2+82d \right)\frac{N^5}{4^3.3!}+O(N^4).$$
\item If $$l>\frac{-4}{147}(18d^3-242d^2+533d)$$
then $X$ has global $3$-jet differentials i.e. global sections of $E_{3,N}\Omega_{(X,\Delta)}$. More precisely
$$h^0(X,E_{3,N}\Omega_{X})\geq \left(\frac{147l}{2}+36d^3-484d^2+1066d \right)\frac{N^7}{6^5}+O(N^6).$$
\end{enumerate}
In particular, a quintic with the maximum number of nodes, 31, has global $3$-jet differentials.
\end{theorem}

\begin{proof}
The proof is just a generalization of the original approach of \cite{GG80} to the orbifold setting. Recall that on a complex manifold $Y$ we have a filtration of $E_{k,m}^{GG}\Omega_{Y}$ whose graded terms are
\begin{equation*}
Gr^{l}(E_{k,N}^{GG}\Omega_{Y})=S^{l_{1}}\Omega_{Y}\otimes S^{l_{2}}\Omega_{Y}\otimes ...\otimes S^{l_{k}}\Omega_{Y},
\end{equation*}
where $l:=(l_{1},l_{2},...,l_{k})\in \mathbb{N}^{k}$ verify $%
l_{1}+2l_{2}+...+kl_{k}=N.$
This enables the following Riemann-Roch computations on surfaces
$$\chi(Y,E_{2,N}^{GG}\Omega_{Y})=(7c_1^2-5c_2)\frac{N^5}{4^3.3!}+O(N^4),$$
$$\chi(Y,E_{3,N}^{GG}\Omega_{Y})=(85c_1^2-49c_2)\frac{N^7}{6^5}+O(N^6).$$
These Riemann-Roch estimations extend to the orbifold setting in the same way as described above for symmetric differentials, providing
$$\chi(\mathcal{X},E_{2,N}^{GG}\Omega_{\mathcal{X}})=(7c_1^2(\mathcal{X})-5c_2(\mathcal{X}))\frac{N^5}{4^3.3!}+O(N^4),$$
$$\chi(\mathcal{X},E_{3,N}^{GG}\Omega_{\mathcal{X}})=(85c_1^2(\mathcal{X})-49c_2(\mathcal{X}))\frac{N^7}{6^5}+O(N^6).$$
To conclude, in the case of manifolds, one applies a vanishing theorem of Bogomolov \cite{Bo79} for the $h^2$. This vanishing theorem extends to the orbifold setting as shown in \cite{BoDe} (proposition 2.3). Then one uses the explicit computations of the orbifold Chern classes described above.
\end{proof}

As a consequence we obtain
\begin{thmc}
Let $X\subset \bP^3$ be a quintic with the maximum number of nodes, 31. Then every classical orbifold entire curve satisfies an algebraic differential equation of order $3$.
\end{thmc}

\bigskip
\noindent \texttt{rousseau@math.u-strasbg.fr}

\noindent D\'{e}partement de Math\'{e}matiques,

\noindent IRMA,\newline
Universit\'{e} de Strasbourg,

\noindent 7, rue Ren\'{e} Descartes,\newline
\noindent 67084 STRASBOURG CEDEX

\noindent FRANCE

\end{document}